\documentclass{amsart}
\usepackage{url}

\usepackage{xcolor}

\newcommand{\domain}{\mathcal{D}}
\newcommand{\set}[1]{\left\{#1\right\}}
\newcommand{\norm}[2]{\| #1 \|_{#2}}

\newcommand{\lr}[1]{\left( #1\right)}
\newcommand{\yd}{y^\delta}
\newcommand{\xp}{x^\dag}
\newcommand{\xad}{x_{\alpha}^\delta}
\newcommand{\xa}{x_{\alpha}}
\newcommand{\xadast}{x_{\aast}^\delta}
\newcommand{\xaast}{x_{\aast}}

\newcommand{\aast}{\alpha_{\ast}}

\newcommand{\mpe}{M_{p,E}}
\newcommand{\Jaa}{T_{-a,\alpha}}
\theoremstyle{plain}
\newtheorem{theorem}{Theorem}
\newtheorem{proposition}{Proposition}

\newtheorem{lemma}{Lemma}
\newtheorem{ass}{Assumption}
\title[A priori parameter choice in Tikhonov regularization]{A priori parameter choice in Tikhonov regularization with
  oversmoothing penalty for non-linear ill-posed problems}
\author{Bernd Hofmann}
\address{Faculty of Mathematics, Chemnitz University of Technology,
  09107 Chemnitz, Germany}
\email{bernd.hofmann@tu-chemnitz.de}
\author{ Peter Math\'e}
\address{Weierstra{\ss} Institute for Applied Analysis and
  Stochastics, Mohrenstra{\ss}e 39, 10117 Berlin, Germany}
\email{peter.mathe@wias-berlin.de}

\begin{document}            

\begin{abstract}
We study Tikhonov regularization for certain classes of non-linear
ill-posed operator equations in Hilbert space. Emphasis is on the case
where the solution smoothness fails to have a finite penalty value, as
in the preceding study {\it Tikhonov regularization with oversmoothing
  penalty for non-linear ill-posed problems in {H}ilbert scales}. Inverse
  Problems 34(1), 2018, by the same authors. Optimal order convergence
  rates are established for the specific a priori parameter choice, as
  used for the corresponding linear equations.
\end{abstract}
  
\maketitle              

\section{Introduction}
\label{sec:intro}

The present paper is closely
related to the recent work~\cite{HofMat18} of the authors published in
the journal {\sl Inverse Problems} devoted to the Tikhonov
regularization for non-linear operator equation with oversmoothing
penalties in Hilbert scales. Here we adopt the model and terminology.
Since ibidem convergence rates of optimal
order were only proven for the discrepancy principle as an a
posteriori choice of the regularization parameter, we try here to
close a gap in regularization theory by extending the same rate
results to the case of appropriate a priori choices. This is in
good coincidence with the corresponding results for linear operator
equations presented in the seminal paper~\cite{Natterer84}, where the
same a priori parameter choice successfully allowed for order optimal
convergence rates also in the case of oversmoothing penalties.

We consider the approximate solution of an operator equation
\begin{equation}
  \label{eq:opeq}
  F(x) =y,
\end{equation}
which models an inverse problem with an (in general) non-linear
forward operator $F: \domain(F) \subseteq X \to Y$ mapping between the
infinite dimensional Hilbert spaces $X$ and $Y$, with domain
$\domain(F)$.  By $\xp$ we denote a solution to~(\ref{eq:opeq}) for
given right-hand side $y$.  As a consequence of the `smoothing'
property of $F$, which is typical for inverse problems, the non-linear
equation \eqref{eq:opeq} is {\sl locally ill-posed} at the solution
point $\xp \in \domain(F)$ (cf.~\cite[Def.~2]{HofSch94}), which in
particular means that stability estimates of the form
\begin{equation} \label{eq:ill} \|x-\xp\|_X \le
  \varphi(\|F(x)-F(\xp)\|_Y)
\end{equation}
cannot hold for all $x \in \domain(F)$ in an arbitrarily small ball
around $\xp$ and for strictly increasing continuous functions
$\varphi: [0,\infty) \to [0,\infty)$ with $\varphi(0)=0$. However,
inequalities similar to \eqref{eq:ill}, called {\sl conditional
  stability estimates}, can hold on the one hand if the admissible
range of $x \in \domain(F)$ is restricted to densely defined subspaces
of $X$. In this context, we refer to the seminal paper
\cite{ChengYam00} as well as to \cite{Chengetal14,HofYam10,HohWei17} and references therein. On the
other
hand, they can hold for all $x \in \domain(F)$ if the term
$\|x-\xp\|_X$ on the left-hand side of the inequality \eqref{eq:ill}
is replaced with a weaker distance measure, for example a weaker norm
(cf.~\cite{EggerHof18} and references therein). In this paper, we
follow a combination of both approaches in a Hilbert scale setting.

Based on noisy data $\yd \in Y$, obeying the deterministic noise model
\begin{equation} \label{eq:noise} \|y-\yd\|_Y \le \delta
\end{equation}
with {\it noise level} $\delta>0$, we use within the domain
$$ \domain:=\domain(F) \cap \domain(B)\neq \emptyset$$
minimizers $\xad \in \domain$ of the Tikhonov functional
$T^\delta_\alpha$ solving the extremal problem
\begin{equation}
  \label{eq:tikhonov}
  T^\delta_{\alpha}(x) := \norm{F(x) - \yd}{Y}^{2} + \alpha\norm{B(x -
    \bar x)}{X}^{2} \to \min, \;\; \mbox{subject to} \;\;
  x\in\domain,
\end{equation}
as stable approximate solutions (regularized solutions) to
$\xp$. Above, the element~$\bar x  \in \domain$ is a given smooth reference element,
and $B\colon \domain(B) \subset X \to X$ is a densely defined,
unbounded, linear self-adjoint operator, which is strictly positive
such that we have for some $m>0$
\begin{equation} \label{eq:m} \norm{Bx}{X}\geq m \norm{x}{X},\quad
  \mbox{for all} \quad x\in \domain(B).
\end{equation}
This operator $B$ generates a {\it Hilbert scale}
$\{X_\tau\}_{\tau \in \mathbb{R}}$ with $X_0=X$,
$X_\tau=\domain(B^\tau)$, and with corresponding norms
$\|x\|_\tau:=\|B^\tau x\|_X$.

Here we shall focus on the case of an oversmoothing penalty, which means that~$\xp\not\in\domain(B)$ such
that~$T^\delta_{\alpha}(\xp)= \infty$. In this case, the
regularizing property
$T^\delta_{\alpha}(\xad) \le T^\delta_{\alpha}(\xp)$ does not provide
additional value here.  Throughout this paper, we assume the operator $F$
to be sequentially weakly continuous and its domain $\domain(F)$ to be
a convex and closed subset of $X$, which makes the general results of
\cite[Section~4.1.1]{SKHK12} on existence and stability of the
regularized solutions $\xad$ applicable.

The paper is organized as follows: In Section~\ref{sec:result} we
formulate non-linearity and smoothness assumptions, which are required
for obtaining a convergence rate result for Tikhonov regularization in
Hilbert scales in the case of oversmoothing penalties under an
appropriate a priori choice of the regularization parameter. Also in
Section~\ref{sec:result} we formulate the main
theorem. Its proof will then follow from two propositions which are
stated.
Section~\ref{sec:aux} is devoted to proving both propositions. Section~\ref{sec:conclude} completes the paper
with some concluding discussions.

\section{Assumptions and main result}
\label{sec:result}
In accordance with the previous study~\cite{HofMat18} we make the
following additional assumption on the structure of non-linearity for
the forward operator $F$ with respect to the Hilbert scale generated
by the operator~$B$. Sufficient conditions and examples for this
non-linearity assumption can be found in the appendix of~\cite{HofMat18}.
\begin{ass}
  [Non-linearity structure]\label{ass:nonlinearity} There is a
  number~$a>0$, and there are constants~$0< c_{a} \le C_{a}<\infty$
  such that
  \begin{equation} \label{eq:twosides} c_{a}\norm{x - \xp}{-a} \leq
    \norm{F(x) - F(\xp)}{Y} \leq C_{a}\norm{x - \xp}{-a} \quad
    \mbox{for all} \;\, x\in \mathcal{D}.
  \end{equation}
\end{ass}
The left-hand inequality of condition \eqref{eq:twosides} implies that, for the
right-hand side $y$, there is no solution to \eqref{eq:opeq} which belongs to $\domain$.  Moreover note that the parameter $a>0$ in
Assumption~\ref{ass:nonlinearity} can be interpreted as {\it degree of ill-posedness} of the mapping~$F$ at $\xp$.

The solution smoothness is measured with respect to the
generator~$B$ of the Hilbert scale as follows. We fix the reference
element~$\bar x\in\domain$,  occurring in
the penalty functional of $T^\delta_\alpha$.
\begin{ass}
  [Solution smoothness]\label{ass:smooth} There are $0<p<1$ and
  $E <\infty$ such that~$\xp\in \domain(B^{p})$ and
  \begin{equation}
    \label{eq:mpe}
    \xp - \bar x \in \mpe:= \set{x\in X_{p},\quad \norm{x}{p}:= \norm{B^{p}x}{X} \leq E}.
  \end{equation}
  Moreover, we assume that $\xp$ is an interior point of $\domain(F)$,
  but $\xp \notin \domain(B)$.
\end{ass}

We shall analyze the error behavior of the minimizer~$\xad$ to the
Tikhonov functional $T^\delta_\alpha$ for the following specific
\emph{a priori} parameter choice.
\begin{ass}
  [A priori parameter choice]\label{ass:parameter} Given noise
  level~$\delta >0$, ill-posedness degree $a>0$
  (cf.~Assumption~\ref{ass:nonlinearity}) and solution smoothness
  $p \in (0,1)$ (cf.~Assumption~\ref{ass:smooth}), let
  \begin{equation} \label{eq:apriori} \aast =\aast(\delta):=
    \delta^{\frac{2(a+1)}{a + p}}.
  \end{equation}
\end{ass}
We shall occasionally use the
identity~$\frac{\delta}{\sqrt{\aast}} = \delta^{\frac{p-1}{a + p}}$,
and we highlight that for this parameter choice we
have~$\frac{\delta}{\sqrt{\aast(\delta)}} \to \infty$ as~$\delta\to 0$, for~$0 < p < 1$.

The main result is as follows.
\begin{theorem} \label{thm:main} Under the assumptions stated above
  let~$\xadast$ be the minimizer of the Tikhonov functional
  $T^\delta_{\alpha_*}$ for the a~priori choice $\alpha_*$ from
  \eqref{eq:apriori}.  Then we have the convergence rate
  \begin{equation}\label{eq:rate}
    \norm{\xadast - \xp}{X} = \mathcal{O} \left( \delta^{\frac p
        {a+p}}\right) \quad \mbox{as} \quad \delta \to 0.
  \end{equation}
\end{theorem}
This asymptotics is an immediate consequence of the following
  two propositions, the proofs of which will be given in the next section.
 \begin{proposition}\label{pro:-a-bound}
    Under the a~priori choice $\alpha_*$ from \eqref{eq:apriori} and
    for sufficiently small $\delta>0$ we have that
    \begin{equation}
      \label{eq:-a-bound}
      \norm{\xadast - \xp}{-a} \leq K \delta
    \end{equation}
    holds for some positive constant $K$.
  \end{proposition}
  \begin{proposition}\label{pro:p-bound}
    Under the a~priori choice $\alpha_*$ from \eqref{eq:apriori} and
    for sufficiently small $\delta>0$ we have that
    \begin{equation}
      \label{eq:p-bound}
      \norm{\xadast - \xp}{p} \leq \tilde E
    \end{equation}
    holds for some positive constant $\tilde E$.
  \end{proposition}

  {\parindent0em \sl Proof of Theorem~\ref{thm:main}}.
   Taking into account the assertions of Propositions~\ref{pro:-a-bound} and
   \ref{pro:p-bound}, the convergence rate
  \eqref{eq:rate} follows directly from the interpolation inequality
  of the Hilbert scale $\{X_{\tau}\}_{\tau \in \mathbb{R}}$, applied
  here in the form
  \begin{equation}
    \label{eq:interpolation}
    \norm{\xadast - \xp}{X} \leq \norm{\xadast - \xp}{-a}^{\frac{p}{a +
        p}} \norm{\xadast - \xp}{p} ^{\frac{a}{a + p}}.
  \end{equation}
  Thus, the proof of the theorem is complete. \qed

\section{Proofs of Propositions~\ref{pro:-a-bound}
  and~\ref{pro:p-bound}}
\label{sec:aux}

Propositions~\ref{pro:-a-bound} and~\ref{pro:p-bound} yield bounds in
the (weak) $(-a)$-norm and in the (strong) $(p)$-norm, respectively.
For the proofs we shall use \emph{auxiliary elements}~$\xaast$,
constructed as follows.
Precisely, in
conjunction with the Tikhonov functional $T_\alpha^\delta$ from
(\ref{eq:tikhonov}) we consider the artificial Tikhonov functional
\begin{equation}
  \label{eq:Tikh-aux}
  \Jaa (x) := \norm{x - \xp}{-a}^{2} + \alpha \norm{B (x - \bar x)}{X}^{2},
\end{equation}
which is well-defined for all $x \in X$. Let~$\xa$ be the minimizers
of~$\Jaa$ over all $x \in X$ that are for all $\alpha>0$ independent
of the noise level $\delta>0$ and recall now the parameter choice
\eqref{eq:apriori} from Assumption~\ref{ass:parameter}. For this
choice of the regularization parameter the estimates
from~\cite[Prop.~2]{HofMat18} yield immediately the following assertions.

\begin{proposition}\label{pro:old-aast}
  Suppose that~$\xp \in \mpe$ for some $0<p<1$, and let~$\xa$ be the
  minimizer of~$\Jaa$.  Given $\aast>0$ as in
  Assumption~\ref{ass:parameter} the resulting element~$\xaast$ obeys
  the bounds
  \begin{align}
    \norm{\xaast - \xp}{X} &\leq E \delta^{p/(a + p)}, \label{it:xa-xp}\\
    \norm{B^{-a}(\xaast - \xp)}{X} &\leq E \delta,\label{it:xa-xp-a}\\
    \norm{B(\xaast - \bar x)}{X} & \leq E \delta^{(p-1)/(a + p)} = E \frac{\delta}{\sqrt{\aast}}, \label{it:B-xa}
                                   \intertext{and}
                                   \Jaa(\xaast) & \leq 2E^{2} \delta.\label{it:Jaa}
  \end{align}
  Moreover, we have that
  \begin{equation}
    \label{eq:xa-p-norm}
    \norm{\xaast - \bar x}{p} \leq E,\  \text{and}\   \norm{\xaast - \xp}{p} \leq E.
  \end{equation}
\end{proposition}

\bigskip

Notice, that in contrast to the solution element~$\xp$ the auxiliary
element~$\xaast$ belongs to~$\domain$, provided that $\delta$ is small
enough, and hence we can use the minimizing property
\begin{equation} \label{eq:regulpro}
  T^{\delta}_{\aast}(x_{\alpha_*}^\delta)\le
  T^{\delta}_{\aast}(\xaast).
\end{equation}
We derive the following
consequence of Proposition~\ref{pro:old-aast} formulated in
Proposition~\ref{pro:tikhonov}.
\begin{proposition}\label{pro:tikhonov}
  Let~$\xadast$ be the minimizer of $T^{\delta}_{\aast}$ for the
  Tikhonov functional $T_\alpha^\delta$ from (\ref{eq:tikhonov}) with
  the choice~$\aast$, as in
  Assumption~\ref{ass:parameter}, of the regularization
  parameter $\alpha>0$. Then we
  have for sufficiently small $\delta>0$ that
  \begin{align}
    \norm{F(\xadast) - \yd}{Y} \leq C \delta,
    \label{it:f-bound1}
    \intertext{and}
    \norm{B(\xadast - \bar x)}{X} \leq C \frac{\delta}{\sqrt{\aast}},
    \label{it:f-bound2}
  \end{align}
  where $C:= \lr{(C_{a}E +1)^{2}+ E^{2}}^{1/2}$.
\end{proposition}

\begin{proof}
 Using~(\ref{eq:regulpro}) it is enough to bound
$T^{\delta}_{\aast}(\xaast)$ by~$C \delta$.
  We first argue  that $\xaast$,
the minimizer of the auxiliary functional~(\ref{eq:Tikh-aux}) for
$\alpha=\alpha_*$,  belongs to the set $\domain(F) \cap \domain(B)$.
Indeed, the bound~(\ref{it:B-xa}) shows that $\xaast \in \domain(B)$,  and
the bound~(\ref{it:xa-xp}) indicates that $\xaast \in \domain(F)$ for
sufficiently small $\delta>0$, where we use that~$\xp$ is an
interior point of $\domain(F)$. Now we find from (\ref{it:xa-xp-a}) \&
(\ref{it:B-xa}),  and from the right-hand inequality of~(\ref{eq:twosides}), that
\begin{align*}
  T^{\delta}_{\aast}(\xaast)  & \leq \lr{\norm{F(\xaast) - F(\xp)}{Y} +
                                \norm{F(\xp) - \yd}{Y}}^{2} + \aast
                                \norm{B(\xaast - \bar x)}{X}^{2}\\
                              &\leq \lr{C_{a} \norm{ \xaast - \xp}{-a} + \delta}^{2} + E^{2}
                                \aast   \delta^{2(p-1)/(a + p)}\\
                              & \leq \lr{C_{a}E\delta + \delta}^{2} + E^{2}\delta^{2}
  \\
                              & = \lr{(C_{a}E +1)^{2} + E^{2}} \delta^{2}.
\end{align*}
This completes the proof of Proposition~\ref{pro:tikhonov}. \qed
\end{proof}

Now we  turn to the proofs of
Propositions~\ref{pro:-a-bound} and~\ref{pro:p-bound}.

\smallskip

{\parindent0em \sl Proof of Proposition~\ref{pro:-a-bound}}.
  Here we use the left-hand side in the non-linearity condition from
Assumption~\ref{ass:nonlinearity} and find for sufficiently small
$\delta>0$
\begin{align*}
  \norm{\xadast - \xp}{-a} & \leq \frac 1 {c_{a}} \norm{F(\xadast) -
                             F(\xp)}{Y}\\
                           & \leq \frac 1 {c_{a}} \lr{  \norm{F(\xadast) -\yd}{Y} +
                             \norm{F(\xp) - \yd}{Y}}\\
                           & \leq \frac 1 {c_{a}} \lr{ C \delta + \delta} = \frac 1 {c_{a}} \lr{ C +1} \delta=K \delta,
\end{align*}
where~$C$ is the constant from Proposition~\ref{pro:tikhonov} and
$K:=\frac 1 {c_{a}} \lr{ C +1}$. The proof is complete. \qed

\medskip

In order to establish the bound occurring in Proposition~\ref{pro:p-bound} we
start with the following estimate.

\begin{lemma}\label{lem:B-bound}
  Let~$\aast$ be as in Assumption~\ref{ass:parameter}. Then
  there is a constant~$\tilde C$ such that  we have
  for sufficiently small $\delta>0$ that
  $$
  \norm{B(\xadast - \xaast)}{X} \leq \tilde C
  \frac{\delta}{\sqrt{\aast}}\,.
  $$
\end{lemma}
\begin{proof}
  By using the triangle inequality we find that
  $$
  \norm{B(\xadast - \xaast)}{X} \leq \norm{B(\xaast - \bar x)}{X} +
  \norm{B(\xadast - \bar x)}{X}
$$
The first summand on the right was bounded in~(\ref{it:B-xa}), and the
second was bounded in Proposition~\ref{pro:tikhonov} for sufficiently
small $\delta>0$. This yields the assertion with~$\tilde C:= C + E$,
where~$C$ is the constant from Proposition~\ref{pro:tikhonov}.  \qed
\end{proof}
This allows for the following result.
\begin{proposition}\label{pro:barE}
  There is a constant~$\bar E$ such that we have for~$\aast$ as in
  Assumption~\ref{ass:parameter} and for sufficiently small $\delta>0$
  that
  $$
  \norm{\xadast - \xaast}{p} \leq \bar E.
  $$
\end{proposition}
\begin{proof}
  Again, we use the interpolation inequality, now in the form of
  $$
  \norm{\xadast - \xaast}{p} \leq \norm{\xadast - \xaast}{1}^{\frac{a
      + p}{a + 1}} \norm{\xadast - \xaast}{-a} ^{\frac{1 - p}{a + 1}}.
$$
The norm in the first factor was bounded in
Lemma~\ref{lem:B-bound}. The norm in the second factor can be bounded
from
\begin{equation} \label{eq:addlast} \norm{\xadast - \xaast}{-a} \leq
  \norm{\xadast - \xp}{-a} + \norm{\xaast - \xp}{-a}.
\end{equation}
Now we discuss both summands in the right-hand side of the inequality
\eqref{eq:addlast}. The first summand was bounded in
Proposition~\ref{pro:-a-bound} by a multiple of~$\delta$, and the same
holds true for the second summand by virtue of~(\ref{it:xa-xp-a}).
Therefore, there is a constant~$\bar E$ such that
$$
\norm{\xadast - \xaast}{p} \leq \bar E
\lr{\frac{\delta}{\sqrt{\aast}}}^{^{\frac{a + p}{a + 1}}}
\delta^{\frac{1 - p}{a + 1}} = \bar E\, \delta\, \aast^{- \frac{a + p
  }{2(a + 1)}}= \bar E.
$$
This completes the proof of Proposition~\ref{pro:barE}.
\end{proof}

We have gathered all auxiliary estimates in order to turn to the final proof.

\smallskip
{\parindent0em \sl Proof of Proposition~\ref{pro:p-bound}}. The
estimate \eqref{eq:p-bound} 
is an
immediate consequence of Proposition~\ref{pro:barE} and of the second
bound from~(\ref{eq:xa-p-norm}), overall yielding the
constant \linebreak $\tilde E := \bar E + E$. This completes the proof.\qed

\section{Conclusions}
\label{sec:conclude}

We have shown that under the non-linearity Assumption~\ref{ass:nonlinearity} and
the solution smoothness given as in Assumption~\ref{ass:smooth} the a priori regularization parameter choice
$$
\alpha_*=\aast(\delta)= \delta^{\frac{2(a+1)}{a +
    p}}=\delta^{2-\frac{2(p-1)}{a+p}}
$$
allows for the order optimal convergence rate \eqref{eq:rate} for all
$0<p<1$.
The obtained rate from Theorem~\ref{thm:main} is valid for~$0
  < p \leq a + 1$ when using the same a priori choice of the regularization parameter
  from Assumption~\ref{ass:parameter}. In all cases, we have
  that~$\aast(\delta) \to 0$ as $\delta \to 0$.    In the regular case
  with~$1 < p \leq a + 1$ we also find the usual
  convergence~$\frac{\delta^2}{\alpha(\delta)} \to 0$ as $\delta \to
  0$. For the borderline
case $p=1$ the quotient $\frac{\delta^2}{\aast(\delta)}$ is constant.
However, in the oversmoothing case~$0 < p < 1$, as considered here, we
find that~$\frac{\delta^2}{\aast(\delta)} \to \infty$ as $\delta \to
0$.

We stress another observation. In the regular case with~$p>1$
  we have the {\it convergence
  property}
\begin{equation} \label{eq:confin} \lim \limits_{\delta \to 0}
  \|\xadast-\xp\|_X=0\,,
\end{equation}
which is a consequence of the sequential weak
compactness of a ball and of the Kadec-Klee property in the Hilbert
space $X$. This cannot be shown for $0<p<1$ without using additional smoothness
of the solution $\xp$. Hence, the convergence property
\eqref{eq:confin} as an implication of Theorem~\ref{thm:main} is
essentially based on the existence of a positive value $p$ in
\eqref{eq:mpe} expressing solution smoothness.

\section*{Acknowledgment}
The research was financially supported by Deutsche
Forschungsgemeinschaft (DFG-grant HO 1454/12-1) and by the Weierstrass
Institute for Applied Analysis and Stochastics, Berlin. 

%
\providecommand{\bysame}{\leavevmode\hbox to3em{\hrulefill}\thinspace}
\providecommand{\MR}{\relax\ifhmode\unskip\space\fi MR }
\providecommand{\MRhref}[2]{%
  \href{http://www.ams.org/mathscinet-getitem?mr=#1}{#2}
}
\providecommand{\href}[2]{#2}

\end{document}